\documentclass[a4paper,11pt]{article}
\usepackage{graphicx} 

\usepackage{amsfonts}
\usepackage{epstopdf}
\usepackage{latexsym}
\usepackage{amsmath}
\usepackage{amsthm}
\usepackage{setspace}
\usepackage{comment}
\usepackage{bm}
\usepackage[caption=false]{subfig}
\usepackage{color}
\usepackage{algorithm}
\usepackage{algpseudocode}
\usepackage{algorithmicx}
\usepackage{amsopn}
\usepackage{amssymb}
\usepackage{mathtools}

\newcommand{\mini}{\mathop{\rm minimize}}

\newtheorem{theorem}{Theorem}
\newtheorem{proposition}{Proposition}
\newtheorem{lemma}{Lemma}

\newtheorem{remark}{Remark}
\newtheorem{example}{Example}

\title{Convergence analysis of a regularized Newton method with generalized regularization terms for convex optimization problems}
\author{Yuya Yamakawa\footnote{Graduate School of Informatics, Kyoto University, Yoshidahommachi, Sakyo-ku, Kyoto-shi, Kyoto 606-8501, Japan, E-mail: yuya@i.kyoto-u.ac.jp} and Nobuo Yamashita\footnote{Graduate School of Informatics, Kyoto University, Yoshidahommachi, Sakyo-ku, Kyoto-shi, Kyoto 606-8501, Japan, Email: nobuo@kyoto-u.ac.jp}}
\date{July 10, 2024}

\begin{document}

\maketitle

{\small
{\bf Abstract.}
This paper presents a regularized Newton method (RNM) with generalized regularization terms for unconstrained convex optimization problems. The generalized regularization includes quadratic, cubic, and elastic net regularizations as special cases. Therefore, the proposed method serves as a general framework that includes not only the classical and cubic RNMs but also a novel RNM with elastic net regularization. We show that the proposed RNM has the global $\mathcal{O}(k^{-2})$ and local superlinear convergence, which are the same as those of the cubic RNM. 
}

\section{Introduction} \label{sec:intro}
We consider the following unconstrained convex optimization problem:
\begin{align} \label{problem:convex}
\begin{aligned}
& \displaystyle \mini_{x \in \mathbb{R}^{n}} & & f(x),
\end{aligned}
\end{align}
where the function $f$ is twice continuously differentiable and convex on $\mathbb{R}^{n}$. 
\par
Newton's method is one of the most well-known and basic iterative methods for solving unconstrained convex optimization problems. Each iteration computes a search direction $d_{k}$, which is a solution of the following subproblem:
\begin{align*}
\begin{aligned}
& \displaystyle \mini_{d \in \mathbb{R}^{n}} & & \langle \nabla f(x_{k}), d \rangle + \frac{1}{2} \langle \nabla^{2} f(x_{k}) d, d \rangle,
\end{aligned}
\end{align*}
and updates the current point $x_{k}$ as $x_{k+1} \coloneqq x_{k} + t_{k} d_{k}$, where $t_{k} > 0$ denotes a step size. It converges rapidly thanks to the use of the second-order information, that is, $\nabla^{2} f(x_{k})$, of the objective function. However, it requires that $\nabla^{2} f(x_{k})$ is nonsingular at each iteration. Even if the Hessian is nonsingular, the convergence rate may be reduced to linear when the Hessian is close to singular. Several variants of Newton's method have been proposed to overcome these drawbacks including regularized Newton methods (RNMs)~\cite{DoMiNe24,LiFuQiYa04,LiLi09,Po09,Mi23}, cubic RNMs~\cite{NePo06,Ne08,YuZhSo19}, and so forth~\cite{GoReBa20,GrLaLu86,RoNe21,CrRo20,DoRi18,GoKoLiRi19,HaDoRiNe20,MaBaRu19,RoKr16,SoMiMoDeGu20}.
\par
RNMs can be considered modifications of Newton's method because they improve the subproblem of Newton's method such that it can be solved even if the Hessian matrix is singular. More precisely, RNMs iteratively solve the following subproblem to find a search direction $d_{k}$:
\begin{align*}
\begin{aligned}
& \displaystyle \mini_{d \in \mathbb{R}^{n}} & & \langle \nabla f(x_{k}), d \rangle + \frac{1}{2} \langle \nabla^{2} f(x_{k}) d, d \rangle + \frac{\mu_{k}}{2} \Vert d \Vert^{2},
\end{aligned}
\end{align*}
where $\mu_{k} > 0$ denotes a parameter. Since the objective function is strongly convex, the subproblem has a unique optimum. Nesterov and Polyak~\cite{NePo06} proposed an RNM with the cubic regularization $\frac{\mu_{k}}{6} \Vert d \Vert^{3}$. The proposed method is called the cubic RNM and iteratively solves the following subproblem:
\begin{align*}
\begin{aligned}
& \displaystyle \mini_{d \in \mathbb{R}^{n}} & & \langle \nabla f(x_{k}), d \rangle + \frac{1}{2} \langle \nabla^{2} f(x_{k}) d, d \rangle + \frac{\mu_{k}}{6} \Vert d \Vert^{3},
\end{aligned}
\end{align*}
For classical and cubic RNMs, global ${\cal O}(k^{-2})$ and local superlinear convergence were proven in~\cite{LiFuQiYa04,Po09,NePo06,Mi23}.
\par
For least squares problems, Ariizumi, Yamakawa, and Yamashita~\cite{ArYaYa24} recently proposed a Levenberg-Marquardt method (LMM) equipped with a generalized regularization term and showed its global and local superlinear convergence. Although a subproblem of the classical LMM has a quadratic regularization term $\frac{\mu_{k}}{2} \Vert d \Vert^{2}$, they generalized the regularization term such that another regularization can be adopted, such as the $L^{1}$ and elastic-net regularization. Moreover, they reported numerical experiments in which their LMM with the elastic-net regularization worked well for certain examples, owing to the sparsity of the search direction.
\par
Inspired by Ariizumi, Yamakawa, and Yamashita~\cite{ArYaYa24}, we propose a generalized RNM (GRNM) for solving problem~\eqref{problem:convex}. We adopt new regularization terms provided as $\frac{\mu_{k}}{p} \Vert d \Vert_{2}^{p} + \rho_{k} \Vert d \Vert_{1}$, where $p \in (1,3]$ is a pre-fixed parameter. With the addition of the new regularization terms, the GRNM includes classical and cubic RNMs as well as novel RNMs with other regularizations, such as the elastic net and so forth. More precisely, if $p=2$ and $\rho_{k} = 0$, the GRNM is reduced to the classical RNM; if $p=3$ and $\rho_{k} = 0$, it is equivalent to the cubic RNM. Moreover, if $p=2$ and $\rho_{k} > 0$, it can be regarded as a novel RNM with the elastic-net regularization.
\par
The contributions of this study are as follows. This study provides
\begin{description}
\item[(i)] the generalized RNM stated above;

\item[(ii)] sufficient conditions of $p$, $\mu_{k}$, and $\rho_{k}$ for which the GRNM has the global $\mathcal{O}(k^{-2})$ convergence;

\item[(iii)] local superlinear convergence under the local error bound condition.
\end{description}
Hence, these contributions include classical and cubic RNMs as special cases and provide a framework for new RNMs such as the elastic-net RNM.
\par
The remainder of this paper is organized as follows. Section~\ref{sec:preliminary} provides a general proposition that plays an important role in the analysis of global $\mathcal{O}(k^{-2})$ convergence. Section~\ref{sec:GRNM} describes the proposed method. Section~\ref{sec:global_convergence} presents global $\mathcal{O}(k^{-2})$ convergence of the proposed method. Section~\ref{sec:local_convergence} proves local and superlinear convergence. Finally, concluding remarks are presented in Section~\ref{sec:conclusion}.
\par
Throughout the paper, we use the following mathematical notation. Let $\mathbb{N}$ be the set of natural numbers (positive integers). For $p \in \mathbb{N}$ and $q \in \mathbb{N}$, the set of real matrices with $p$ rows and $q$ columns is denoted by $\mathbb{R}^{p \times q}$. Note that $\mathbb{R}^{p \times 1}$ is equal to the set of $p$-dimensional real vectors, that is, $\mathbb{R}^{p \times 1} = \mathbb{R}^{p}$, and note that $\mathbb{R}^{1}$ represents the set of real numbers, namely, $\mathbb{R}^{1} = \mathbb{R}$. For any $w \in \mathbb{R}^{p}$, the transposition of $w$ is represented as $w^{\top} \in \mathbb{R}^{1 \times p}$. For $u \in \mathbb{R}^{p}$ and $v \in \mathbb{R}^{p}$, the inner product of $u$ and $v$ is defined by $\left\langle u, v \right\rangle \coloneqq u^{\top} v$. We denote by $I$ the identity matrix, where these dimensions are defined by the context. For each $w \in \mathbb{R}^{p}$, the Euclidean and $L^{1}$ norms of $w$ are respectively defined by $\Vert w \Vert \coloneqq \sqrt{\left\langle w, w \right\rangle}$ and $\Vert w \Vert_{1} \coloneqq |[w]_{1}| + |[w]_{2}| + \cdots + |[w]_{p}|$, where $[w]_{j}$ represents the $j$-th element of $w$. For $W \in \mathbb{R}^{p \times q}$, we denote by $\Vert W \Vert$ the operator norm of $W$, that is, $\Vert W \Vert \coloneqq \sup \{ \Vert W u \Vert ; \Vert u \Vert \leq 1 \}$.
Let $\varphi$ be a function from $\mathbb{R}^{p}$ to $\mathbb{R}$. The gradient of $\varphi$ at $w \in \mathbb{R}^{p}$ is represented as $\nabla \varphi(w)$. The Hessian of $\varphi$ at $w \in \mathbb{R}^{p}$ is denoted by $\nabla^{2} \varphi(w)$. For a convex function $\phi \colon \mathbb{R}^{p} \to \mathbb{R}$, we denote by $\partial \phi(w)$ the subdifferential of $\phi$ at $w$. For $\eta \in \mathbb{R}^{p}$ and $r > 0$, we define $B(\eta, r) \coloneqq \{ \mu \in \mathbb{R}^{p}; \Vert \mu - \eta \Vert \leq r \}$. For infinite sequences $\{ a_{k} \} \subset \mathbb{R}$ and $\{ b_{k} \} \subset \mathbb{R}$, we write $a_{k} = {\cal O}(b_{k})~(k \to \infty)$ if there exist $c > 0$ and $n \in \mathbb{N}$ such that $| a_{k} | \leq c | b_{k} |$ for all $k \geq n$.

\section{Preliminaries} \label{sec:preliminary}

This section presents a general proposition that provides sufficient conditions under which arbitrary sequences generated by optimization methods have global ${\cal O}(k^{-2})$ convergence. This proposition plays a critical role in Section~\ref{sec:global_convergence}. The proof of the proposition is inspired by the technique presented in~\cite[Theorem~1]{Mi23}. 

\begin{proposition} \label{prop:global_convergence}
Let $f \colon \mathbb{R}^{n} \to \mathbb{R}$ be a continuously differentiable convex function with a minimum $x^{\ast} \in \mathbb{R}^{n}$.  Let $\{ x_{k} \}$ be an infinite sequence in $\mathbb{R}^{n}$. Suppose that
\begin{description}
\item[{\rm (i)}] $f(x_{k+1}) \leq f(x_{k})$ for all $k \in \mathbb{N} \cup \{ 0 \}$;

\item[{\rm (ii)}] there exists $\gamma > 0$ such that $\Vert \nabla f(x_{k+1}) \Vert \leq \gamma \Vert \nabla f(x_{k}) \Vert$ for all $k \in \mathbb{N} \cup \{ 0 \}$;

\item[{\rm (iii)}] there exists $\delta > 0$ such that $f(x_{k}) - f^{\ast} \leq \delta \Vert \nabla f(x_{k}) \Vert$ for all $k \in \mathbb{N} \cup \{ 0 \}$;

\item[{\rm (iv)}] there exist $\theta \in (0, 1)$, $\nu > 0$, and $\ell \in \mathbb{N}$ such that $\theta^{k} \leq \nu k^{-2}$ and $(\gamma \theta)^{\frac{k}{2}} \leq \nu k^{-2}$ for all $k \geq \ell$, where $\gamma$ is given in~{\rm (ii)}.
\end{description}
Let ${\cal I}(\theta) \coloneqq \{ i \in \mathbb{N} \cup \{ 0 \}; \theta \Vert \nabla f(x_{i}) \Vert \leq \Vert \nabla f(x_{i+1}) \Vert \}$.
Suppose also that
\begin{description}
\item[{\rm (v)}] there exists $\tau > 0$ such that $f(x_{k+1}) - f(x_{k}) \leq - \tau( f(x_{k}) - f^{\ast} )^{\frac{3}{2}}$ for all $k \in {\cal I}(\theta)$.
\end{description}
Then, one of the following statements holds:
\begin{description}
\item[{\rm (a)}] If $| {\cal I}(\theta) | < \infty$ holds, then
\begin{align*}
f(x_{k}) - f^{\ast} \leq \frac{\theta^{-(\widehat{i}+1)} \nu \delta \Vert \nabla f(x_{\widehat{i} + 1}) \Vert}{k^{2}} \quad \forall k \geq \max \{ \ell, \widehat{i}+2 \},
\end{align*}
where $\widehat{i}$ is the largest element of ${\cal I}(\theta)$.

\item[{\rm (b)}] If $| {\cal I}(\theta) | = \infty$ holds, then
\begin{align*}
f(x_{k}) - f^{\ast} \leq \max \left\{ \frac{36 \tau^{-2}}{(k + 4)^{2}}, \frac{ \nu \delta \Vert \nabla f(x_{0}) \Vert }{k^{2}} \right\} \quad \forall k \geq \ell.
\end{align*}
\end{description}
\end{proposition}

\begin{proof}
Let $i_{\ell} \in \mathbb{N}$ be the $\ell$-th smallest element of ${\cal I}(\theta)$, that is, ${\cal I}(\theta) = \{ i_{1}, i_{2}, \ldots, \widehat{i} \}$ with $i_{\ell} < i_{\ell+1}$. Moreover, regarding assumption~(ii), we suppose $\gamma \geq 1$ without loss of generality.
\par
To begin with, we consider case~(a), that is, $| {\cal I}(\theta)| < \infty$ is satisfied. Let ${\cal J}(\theta) \coloneqq \{ i \in \mathbb{N} \cup \{ 0 \}; \theta \Vert \nabla f(x_{i}) \Vert > \Vert \nabla f(x_{i+1}) \Vert \}$. Note that $\widehat{i} = i_{|{\cal I}(\theta)|}$. We can easily observe that
\begin{align}
k \in {\cal J}(\theta) \quad \forall k > \widehat{i}. \label{k_J_infty}
\end{align}
For every $k \geq \widehat{i} + 2$, it follows from~\eqref{k_J_infty} that $j \in {\cal J}(\theta)$ for any $j \in \{ \widehat{i} + 1, \widehat{i} + 2, \ldots, k-1 \}$, that is,
\begin{align}
\Vert \nabla f(x_{k}) \Vert < \theta \Vert \nabla f(x_{k-1}) \Vert < \cdots < \theta^{k-(\widehat{i}+1)} \Vert \nabla f(x_{\widehat{i}+1}) \Vert. \label{ineq:nab_ff}
\end{align}
Let us take $k \geq \max \{ \ell, \widehat{i}+2 \}$ arbitrarily. By assumption~(iii) and \eqref{ineq:nab_ff}, we obtain
\begin{align}
f(x_{k}) - f^{\ast} \leq \delta \Vert \nabla f(x_{k}) \Vert < \theta^{k-(\widehat{i}+1)} \delta \Vert \nabla f(x_{\widehat{i}+1}) \Vert. \label{ineq:case_a}
\end{align}
Recall that $\theta^{k} \leq \nu k^{-2}$ from assumption~(iv). Thus, the desired inequality is derived from~\eqref{ineq:case_a}.
\par
Next, we discuss the case where $|{\cal I}(\theta)| = \infty$ holds. Let $k \in \mathbb{N} \cup \{ 0 \}$ and define $\psi_{k} \coloneqq \tau^{2}(f(x_{i_{k}}) - f^{\ast})$. Combining assumption~(i) and $i_{k+1} \geq i_{k} + 1$ yields
\begin{align}
\psi_{k+1} = \tau^{2} (f(x_{i_{k+1}}) - f^{\ast}) \leq \tau^{2} ( f(x_{i_{k} + 1}) - f^{\ast} ). \label{ineq:ff_ast}
\end{align}
Since assumption~(v) implies that $f(x_{i_{k}+1}) - f(x_{i_{k}}) \leq -\tau (f(x_{i_{k}}) - f^{\ast})^{\frac{3}{2}}$,
\begin{align}
\tau^{2} (f(x_{i_{k}+1}) - f^{\ast}) \leq  \tau^{2} (f(x_{i_{k}}) - f^{\ast}) - \tau^{3} (f(x_{i_{k}}) - f^{\ast})^{\frac{3}{2}} = \psi_{k} - \psi_{k}^{\frac{3}{2}}. \label{ineq:alpha_k}
\end{align}
Exploiting \eqref{ineq:ff_ast} and \eqref{ineq:alpha_k} derives $\psi_{k+1} \leq \psi_{k} - \psi_{k}^{\frac{3}{2}} \leq \psi_{k} - \frac{2}{3} \psi_{k}^{\frac{3}{2}}$. It then follows from \cite[Proposition~1]{Mi23} that $\psi_{k} \leq 9(k+2)^{-2}$, that is,
\begin{align}
f(x_{i_{k}}) - f^{\ast} \leq \frac{9 \tau^{-2}}{(k+2)^{2}} \quad \forall k \in \mathbb{N} \cup \{ 0 \}. \label{ineq:fxik}
\end{align}
Let ${\cal I}_{k} \coloneqq \{ i \in {\cal I}(\theta); i \leq k \}$. In the following, we assume $k \geq \ell$. There are two possible cases: Case~(1) $|{\cal I}_{k}| \geq \frac{k}{2}$ and Case~(2) $|{\cal I}_{k}| < \frac{k}{2}$. 
\begin{description}
\item[{\rm Case~(1):}] The largest element of ${\cal I}_{k}$ can be represented by $i_{|{\cal I}_{k}|}$, and thus $i_{|{\cal I}_{k}|} \leq k$. This fact, assumption~(i), and~\eqref{ineq:fxik} imply that
\begin{align*}
f(x_{k}) - f^{\ast} \leq f(x_{i_{|{\cal I}_{k}|}}) - f^{\ast} \leq \frac{9 \tau^{-2}}{(|{\cal I}_{k}|+2)^{2}} \leq \frac{36 \tau^{-2}}{(k + 4)^{2}}.
\end{align*}

\item[{\rm Case~(2):}] From assumption~(ii), each $j \in \{ 0, 1, \ldots, k-1 \}$ satisfies
\begin{align} \label{ineq:nabf_gamma_delta}
\begin{aligned}
& \Vert \nabla f(x_{j+1}) \Vert \leq \gamma \Vert \nabla f(x_{j}) \Vert & & {\rm if} ~ j \in {\cal I}_{k-1},
\\
& \Vert \nabla f(x_{j+1}) \Vert < \theta \Vert \nabla f(x_{j}) \Vert & & {\rm if} ~ j \not\in {\cal I}_{k-1}.
\end{aligned}
\end{align}
Combining assumption~(iii) and \eqref{ineq:nabf_gamma_delta} derives
\begin{align}
f(x_{k}) - f^{\ast} \leq \delta \Vert \nabla f(x_{k}) \Vert \leq \delta \gamma^{|{\cal I}_{k-1}|} \theta^{k-|{\cal I}_{k-1}|} \Vert \nabla f(x_{0}) \Vert. \label{ineq:fkfastf0}
\end{align}
Note that $\gamma \geq 1$, $\theta \in (0, 1)$, and $(\gamma \theta)^{\frac{k}{2}} \leq \nu k^{-2}$ hold from assumption~(iv) and $k \geq \ell$. It then follows from $|{\cal I}_{k-1}| \leq |{\cal I}_{k}| < \frac{k}{2}$ that
\begin{align}
\gamma^{|{\cal I}_{k-1}|} \theta^{k-|{\cal I}_{k-1}|} \leq (\gamma \theta)^{\frac{k}{2}} \leq \frac{\nu}{k^{2}}. \label{ineq:gamma_delta}
\end{align}
We have from \eqref{ineq:fkfastf0} and \eqref{ineq:gamma_delta} that
\begin{align*}
f(x_{k}) - f^{\ast} \leq \frac{\nu \delta \Vert \nabla f(x_{0}) \Vert}{k^{2}}.
\end{align*}
\end{description}
Cases~(1) and (2) guarantee that the desired inequality holds when $| {\cal I}(\theta) | = \infty$. Therefore, the assertion is proven.
\end{proof}

\begin{remark}
We discuss sufficient conditions for assumptions~{\rm (i)--(iv)} of Proposition~{\rm \ref{prop:global_convergence}}. Assumptions~{\rm (i)} and {\rm (ii)} would be satisfied for any sequence generated by descent methods. Note that $\gamma$ of item~{\rm (ii)} is allowed to be greater than or equal to $1$. Assumption~{\rm (iii)} is satisfied when $\{ x_{k} \}$ is bounded. Assumption~{\rm (iv)} holds if $\theta \in (0,\gamma^{-1})$ because it implies $\theta^{k} = {\cal O}(k^{-2})$ and $(\gamma \theta)^{\frac{k}{2}} = {\cal O}(k^{-2})$ as $k \to \infty$. From these discussions, we can see that assumption~{\rm (v)} is the key to global ${\cal O}(k^{-2})$ convergence.
\end{remark}

\section{An RNM with generalized regularization terms} \label{sec:GRNM}
In this paper, we consider an RNM with generalized regularization terms that iteratively solves the following subproblem:
\begin{align} \label{subproblem}
\begin{aligned}
& \displaystyle \mini_{d \in \mathbb{R}^{n}} & & \langle \nabla f(x_{k}), d \rangle + \frac{1}{2} \langle \nabla^{2} f(x_{k}) d, d \rangle + \frac{\mu_{k}}{p} \Vert d \Vert^{p} + \rho_{k} \Vert d \Vert_{1},
\end{aligned}
\end{align}
where $\mu_{k} > 0$ and $\rho_{k} \geq 0$ are parameters, and $p \in (1,3]$ is a pre-fixed constant. The proposed method obtains a search direction $d_{k}$ by solving subproblem~\eqref{subproblem}, and the point $x_{k}$ is updated as $x_{k+1} \coloneqq x_{k} + t_{k} d_{k}$, where $t_{k}$ is a step size.
\par
Now, we denote by $\varphi_{k}$ the objective function of subproblem~\eqref{subproblem}, namely,
\begin{align*}
\varphi_{k}(d) \coloneqq \langle \nabla f(x_{k}), d \rangle + \frac{1}{2} \langle \nabla^{2} f(x_{k}) d, d \rangle + \frac{\mu_{k}}{p} \Vert d \Vert^{p} + \rho_{k} \Vert d \Vert_{1}.
\end{align*}
Since the objective function $\varphi_{k}$ has generalized regularization terms $\frac{\mu_{k}}{p} \Vert d \Vert^{p}$ and $\rho_{k} \Vert d \Vert_{1}$, we call the proposed method a generalized RNM (GRNM). 

\begin{remark}
The GRNM includes the quadratic, cubic, and elastic net regularization as special cases. Moreover, it includes a novel regularization in addition to the aforementioned regularization.
\end{remark}

\begin{remark}
The proposed GRNM can adopt the $L^{1}$ regularization term, that is, $\mu_{k} = 0$ and $\rho_{k} > 0$. However, subproblem~\eqref{subproblem} with $\mu_{k} = 0$ might have no global optimum when $\nabla^{2} f(x_{k})$ is not positive definite. Conversely, if $\rho_{k}$ is sufficiently large, the solution becomes $0$. We provide sufficient conditions under which \eqref{subproblem} has nonzero solutions.
\end{remark}

\begin{lemma} \label{lemma:solvable}
Let $x_{k} \in \mathbb{R}^{n}$, $\mu_{k} > 0$, and $\rho_{k} \geq 0$ be given. If $\rho_{k} < \Vert \nabla f(x_{k}) \Vert_{\infty}$, then problem~\eqref{subproblem} has a unique global optimum $d_{k} \not= 0$ that satisfies
\begin{align*}
\nabla f(x_{k}) + (\nabla^{2} f(x_{k}) + \mu_{k} \Vert d_{k} \Vert^{p-2} I) d_{k} + \rho_{k} \eta_{k} = 0
\end{align*}
for some $\eta_{k} \in \partial \Vert d_{k} \Vert_{1}$. Moreover, $d_{k}$ is the descent direction of $f$ at $x_{k}$, that is, $\langle \nabla f(x_{k}), d_{k} \rangle < 0$.
\end{lemma}

\begin{proof}
First, we show the solvability of~\eqref{subproblem}. Recall that $\varphi_{k}$ is closed, proper, and coercive. Hence, by using~\cite[Proposition~3.2.1]{Be:ConOpt}, problem~\eqref{subproblem} has a global optimum $d_{k}$. The uniqueness of $d_{k}$ is derived from the strict convexity of $\varphi_{k}$.
\par
Hereafter, we show that $d_{k} \not = 0$ is satisfied. We assume to the contrary that $d_{k} = 0$ holds. As $d_{k} = 0$ satisfies the first-order optimality condition of~\eqref{subproblem}, there exists $\eta_{k} \in \partial \Vert d_{k} \Vert_{1}$ such that $\nabla f(x_{k}) + \rho_{k} \eta_{k} = 0$. It then follows from $\rho_{k} < \Vert \nabla f(x_{k}) \Vert_{\infty}$ that $\Vert \nabla f(x_{k}) \Vert_{\infty} = \rho_{k} < \Vert \nabla f(x_{k}) \Vert_{\infty}$. This result contradicts, that is, $d_{k} \not = 0$.
\par
Finally, the first-order optimality condition of~\eqref{subproblem} leads to the desired equality, and it yields
\begin{align*}
\langle \nabla f(x_{k}), d_{k} \rangle = - \langle (\nabla^2 f(x_{k}) + \mu_{k} \Vert d_{k} \Vert^{p-2} I)d_{k}, d_{k} \rangle - \rho_{k} \Vert d_{k} \Vert_{1} < 0,
\end{align*}
where note that $\langle d_{k}, \eta_{k} \rangle = \Vert d_{k} \Vert_{1}$ and $d_{k} \not = 0$. This completes the proof.
\end{proof}

\begin{remark}
By utilizing the line search strategy or an appropriate choice of $\mu_{k}$ and $\rho_{k}$, we can prove the global convergence of Algorithm~{\rm \ref{algo:GRNM}}. However, because ${\cal O}(k^{-2})$ convergence implies global convergence, we omit discussions on the line search.
\end{remark}

\noindent
We provide a formal description of the proposed method in Algorithm~\ref{algo:GRNM}.
\begin{algorithm}[h] \caption{(GRNM)} \label{algo:GRNM}
\begin{algorithmic}[1]
\State Choose $p \in (1,3]$, $x_{0} \in \mathbb{R}^{n}$, $\varepsilon > 0$, and set $k \coloneqq 0$.
\State If $\Vert \nabla f(x_{k}) \Vert \leq \varepsilon$, then stop. \label{state:TerminationCriteria}
\State Set parameters $\mu_{k} > 0$ and $\rho_{k} \geq 0$, and find a global optimum $d_{k}$ of~\eqref{subproblem}.
\State Set $x_{k+1} \coloneqq x_{k} + d_{k}$.
\State Set $k \leftarrow k+1$, and go to Line~\ref{state:TerminationCriteria}.
\end{algorithmic}
\end{algorithm}

\section{Global $\mathcal{O}(k^{-2})$ convergence of Algorithm~\ref{algo:GRNM}} \label{sec:global_convergence}
This section shows that Algorithm~\ref{algo:GRNM} globally converges with the $\mathcal{O}(k^{-2})$ rate. From now on, we denote by $x^{\ast} \in \mathbb{R}^{n}$ an optimal solution of problem~\eqref{problem:convex}, and use the following notation: $f^{\ast} \coloneqq f(x^{\ast})$ and ${\cal S} \coloneqq \{ x \in \mathbb{R}^{n}; f(x) \leq f(x_{0}) \}$. 
\par
We will show global ${\cal O}(k^{-2})$ convergence of Algorithm~\ref{algo:GRNM} by showing that assumptions~(i)--(v) in Proposition~\ref{prop:global_convergence} hold for a sequence $\{ x_{k} \}$ generated by Algorithm~\ref{algo:GRNM}.
\par
In the subsequent argument, we suppose that $\varepsilon = 0$ and Algorithm~\ref{algo:GRNM} generates an infinite sequence $\{ x_{k} \}$ satisfying $\nabla f(x_{k}) \not = 0$ for each $k \in \mathbb{N} \cup \{ 0 \}$. Moreover, we make the following assumptions.
\begin{description}
\item[(A1)] There exists $L > 0$ such that for any $x, \, y \in \mathbb{R}^{n}$,
\begin{align*}
& \hspace{-5mm} \Vert \nabla f(x) - \nabla f(y) - \nabla^{2} f(y) (x - y) \Vert \leq L \Vert x - y \Vert^{2}, 
\\
& \hspace{-5mm} | f(x) - f(y) - \langle \nabla f(y), x - y \rangle - \frac{1}{2} \langle \nabla^{2} f(y) (x-y), x-y \rangle | \leq \frac{L}{3} \Vert x - y \Vert^{3}.
\end{align*}

\item[(A2)] The parameters $\mu_{k}$ and $\rho_{k}$ are set as follows: For all $k \in \mathbb{N} \cup \{ 0 \}$,
\begin{align*}
\mu_{k} \coloneqq c_{1}^{\frac{p-1}{2}} \Vert \nabla f(x_{k}) \Vert^{\frac{3-p}{2}}, ~ \rho_{k} \coloneqq \min \left\{ \frac{q}{\sqrt{n}} \Vert \nabla f(x_{k}) \Vert, c_{2} \Vert \nabla f(x_{k}) \Vert^{\frac{p+1}{2}} \right\},
\end{align*}
where $c_{1} \geq L$, $c_{2} \in (0,1)$, and $q \geq 0$.

\item[(A3)] There exists $R > 0$ such that ${\cal S} \subset B(0, R)$. 
\end{description}

\noindent
Note that subproblem~\eqref{subproblem} has a global optimum $d_{k} \not= 0$ because (A2) satisfies the condition of Lemma~\ref{lemma:solvable}. Note also that several basic properties of linear algebra derive
\begin{gather}
\left\Vert (\nabla^{2} f(x_{k}) + \mu_{k} \Vert d_{k} \Vert^{p-2} I)^{-1} \right\Vert \leq \mu_{k}^{-1} \Vert d_{k} \Vert^{2-p}, \label{two_ineq:mat_H}
\\
\left\Vert (\nabla^{2} f(x_{k}) + \mu_{k} \Vert d_{k} \Vert^{p-2} I)^{-1} \nabla^{2} f(x_{k}) \right\Vert \leq 1. \label{two_ineq:mat_H2}
\end{gather}

\noindent
We now show that assumptions~(i) and~(ii) in Proposition~\ref{prop:global_convergence} hold.

\begin{lemma} \label{lemma:f_gradf_ineq}
Suppose that {\rm (A1)} and {\rm (A2)} are satisfied. Suppose also that $3 - (1 + q)^{\frac{3-p}{p-1}} > 0$ where $q$ is a constant in {\rm (A2)}. For any $k \in \mathbb{N} \cup \{ 0 \}$, the following inequalities hold:
\begin{align*}
&\hspace{-15mm} \mbox{{\rm (a)}} ~ \displaystyle f(x_{k+1}) - f(x_{k}) \leq -\frac{3 - (1 + q)^{\frac{3-p}{p-1}}}{3} \mu_{k} \Vert d_{k} \Vert^{p} < 0,
\\
&\hspace{-15mm} \mbox{{\rm (b)}} ~ \displaystyle \Vert \nabla f(x_{k+1}) \Vert \leq \left( 1 + (1 + q)^{\frac{3-p}{p-1}} \right) \mu_{k} \Vert d_{k} \Vert^{p-1} 
\\
&\hspace{-15mm}\qquad \qquad \qquad \qquad \qquad + \min \left\{ q \Vert \nabla f(x_{k}) \Vert, \sqrt{n} c_{2} \Vert \nabla f(x_{k}) \Vert^{\frac{p+1}{2}} \right\},
\\
&\hspace{-15mm} \mbox{{\rm (c)}} ~ \displaystyle \Vert \nabla f(x_{k+1}) \Vert \leq \left( 1 + 2 q + (1 + q)^{\frac{2}{p-1}} \right) \Vert \nabla f(x_{k}) \Vert.
\end{align*}
\end{lemma}

\begin{proof}
It follows from Lemma~\ref{lemma:solvable} and \eqref{two_ineq:mat_H} that
\begin{align*}
\Vert d_{k} \Vert &\leq \left\Vert (\nabla^{2} f(x_{k}) + \mu_{k} \Vert d_{k} \Vert^{p-2} I )^{-1} (\nabla f(x_{k}) + \rho_{k} \eta_{k}) \right\Vert \nonumber
\\
&\leq \mu_{k}^{-1} \Vert d_{k} \Vert^{2-p} (\Vert \nabla f(x_{k}) \Vert + \sqrt{n} \rho_{k})
\\
&\leq \mu_{k}^{-1} \Vert d_{k} \Vert^{2-p} (1 + q) \Vert \nabla f(x_{k}) \Vert, 
\end{align*}
that is,
\begin{align}
\Vert d_{k} \Vert \leq \frac{ (1 + q)^{\frac{1}{p-1}} }{\sqrt{c_{1}}} \sqrt{\Vert \nabla f(x_{k}) \Vert}.  \label{ineq:dk}
\end{align}
By~\eqref{ineq:dk} and the first equality of (A2), we have
\begin{align}
L \Vert d_{k} \Vert^{2}
&\leq c_{1} \Vert d_{k} \Vert^{3-p} \cdot \Vert d_{k} \Vert^{p-1} \nonumber
\\
&\leq (1 + q)^{\frac{3-p}{p-1}} \cdot c_{1}^{\frac{p-1}{2}} \Vert \nabla f(x_{k}) \Vert^{\frac{3-p}{2}} \cdot \Vert d_{k} \Vert^{p-1} \nonumber
\\
&= (1 + q)^{\frac{3-p}{p-1}} \mu_{k} \Vert d_{k} \Vert^{p-1}. \label{ineq:c1dk}
\end{align}
Recall that $x_{k+1} = x_{k} + d_{k}$ and $\langle \eta_{k}, d_{k} \rangle = \Vert d_{k} \Vert_{1}$. Combining the second inequality of (A1), Lemma~\ref{lemma:solvable}, and \eqref{ineq:c1dk} yields
\begin{align*}
\begin{aligned}
f(x_{k+1}) - f(x_{k})
&\leq \langle \nabla f(x_{k}) + \nabla^{2} f(x_{k}) d_{k}, d_{k} \rangle - \frac{1}{2} \langle \nabla^{2} f(x_{k}) d_{k}, d_{k} \rangle + \frac{L}{3} \Vert d_{k} \Vert^{3}
\\
&\leq - \mu_{k} \Vert d_{k} \Vert^{p} - \rho_{k} \Vert d_{k} \Vert_{1} + \frac{1}{3} \Vert d_{k} \Vert \cdot L \Vert d_{k} \Vert^{2}
\\
&\leq -\frac{ 3 - (1 + q)^{\frac{3-p}{p-1}} }{3} \mu_{k} \Vert d_{k} \Vert^{p}.
\end{aligned}
\end{align*}
From $x_{k+1} = x_{k} + d_{k}$ and Lemma~\ref{lemma:solvable}, we obtain
\begin{align*}
\nabla f(x_{k+1}) = \nabla f(x_{k} + d_{k}) - \nabla f(x_{k}) - (\nabla^{2} f(x_{k}) + \mu_{k} \Vert d_{k} \Vert^{p-2} I) d_{k} - \rho_{k} \eta_{k}.
\end{align*}
Subsequently, exploiting the first inequality of (A1), the second equality of (A2), and \eqref{ineq:c1dk} derives
\begin{align}
\hspace{-0.5mm} \Vert \nabla f(x_{k+1}) \Vert
&\leq \Vert \nabla f(x_{k} + d_{k}) - \nabla f(x_{k}) - \nabla^{2} f(x_{k}) d_{k} \Vert + \mu_{k} \Vert d_{k} \Vert^{p-1} + \sqrt{n} \rho_{k} \nonumber
\\
&\leq L \Vert d_{k} \Vert^{2} + \mu_{k} \Vert d_{k} \Vert^{p-1} \nonumber
\\
& \qquad \qquad \qquad + \min \left\{ q \Vert \nabla f(x_{k}) \Vert, \sqrt{n} c_{2} \Vert \nabla f(x_{k}) \Vert^{\frac{p+1}{2}} \right\} \nonumber
\\
&\leq \left( 1 + (1 + q)^{\frac{3-p}{p-1}} \right) \mu_{k} \Vert d_{k} \Vert^{p-1} \nonumber
\\
& \qquad \qquad \qquad + \min \left\{ q \Vert \nabla f(x_{k}) \Vert, \sqrt{n} c_{2} \Vert \nabla f(x_{k}) \Vert^{\frac{p+1}{2}} \right\}, \label{ineq:nablaf_kplus1}
\end{align}
namely, item~(b) is verified. Meanwhile, from \eqref{ineq:dk} and the first equality of (A2), we have $\mu_{k} \Vert d_{k} \Vert^{p-1} \leq (1 + q) \Vert \nabla f(x_{k}) \Vert$. Utilizing this result, \eqref{ineq:nablaf_kplus1}, and $\min \{ q \Vert \nabla f(x_{k}) \Vert, \sqrt{n} c_{2} \Vert \nabla f(x_{k}) \Vert^{\frac{p+1}{2}} \} \leq q \Vert \nabla f(x_{k}) \Vert$ means
\begin{align*}
\Vert \nabla f(x_{k+1}) \Vert \leq \left( 1 + 2 q + (1 + q)^{\frac{2}{p-1}} \right) \Vert \nabla f(x_{k}) \Vert. 
\end{align*}
Therefore, the desired inequalities are obtained.
\end{proof}

\noindent
Now, we provide the global ${\cal O}(k^{-2})$ convergence property of Algorithm~\ref{algo:GRNM} by showing assumptions~(iii)-(v) in Proposition~\ref{prop:global_convergence}.

\begin{theorem} \label{th:global}
Suppose that {\rm (A1)--(A3)} hold. Moreover, suppose that the following assumptions~{\rm (A4)--(A6)} hold:
\begin{description}
\item[{\rm (A4)}] $3 - (1 + q)^{\frac{3-p}{p-1}} > 0$;

\item[{\rm (A5)}] $1 + 2 q + (1 + q)^{\frac{2}{p-1}} \geq 1$;

\item[{\rm (A6)}] there exist $\theta \in (q, 1)$, $\nu > 0$, and $\ell \in \mathbb{N}$ such that
\begin{align*}
\theta^{k} \leq \nu k^{-2}, \quad \left( \left( 1 + 2 q + (1 + q)^{\frac{2}{p-1}} \right) \theta \right)^{\frac{k}{2}} \leq \nu k^{-2} \quad \forall k \geq \ell.
\end{align*}
\end{description}
Let ${\cal I}(\theta) \coloneqq \{ i \in \mathbb{N} \cup \{ 0 \}; \theta \Vert \nabla f(x_{i}) \Vert \leq \Vert \nabla f(x_{i+1}) \Vert \}$. Let $D$ and $\tau$ be defined as
\begin{align*}
D \coloneqq R + \Vert x^{\ast} \Vert, \quad \tau \coloneqq \frac{(\theta - q)^{\frac{p}{p-1}} \left( 3 - (1 + q)^{\frac{3-p}{p-1}} \right)}{3 \sqrt{c_{1}} D^{\frac{3}{2}} \left( 1 + (1 + q)^{\frac{3-p}{p-1}} \right)^{\frac{p}{p-1}}}.
\end{align*}
Then, a sequence $\{ x_{k} \}$ generated by Algorithm~{\rm \ref{algo:GRNM}} satisfies one of the following statements:
\begin{description}
\item[{\rm (a)}] If $| {\cal I}(\theta) | < \infty$, then
\begin{align*}
f(x_{k}) - f^{\ast} \leq \frac{\theta^{-(\widehat{i}+1)} \nu D \Vert \nabla f(x_{\widehat{i} + 1}) \Vert}{k^{2}} \quad \forall k \geq \max \{ \ell, \widehat{i}+2 \},
\end{align*}
where $\widehat{i}$ is the largest element of ${\cal I}(\theta)$.

\item[{\rm (b)}] If $| {\cal I}(\theta) | = \infty$, then
\begin{align*}
f(x_{k}) - f^{\ast} \leq \max \left\{ \frac{36 \tau^{-2}}{(k + 4)^{2}}, \frac{ \nu D \Vert \nabla f(x_{0}) \Vert }{k^{2}} \right\} \quad \forall k \geq \ell.
\end{align*}
\end{description}
\end{theorem}

\begin{proof}
If items~(i)--(v) in Proposition~\ref{prop:global_convergence} hold, then the desired result can be obtained. Item~(i) directly follows from~(a) of Lemma~\ref{lemma:f_gradf_ineq}. Let us define $r \coloneqq 1 + 2 q + (1 + q)^{\frac{2}{p-1}}$. Recall that (A5) ensures $r \geq 1$. Thus, item~(ii) holds from (c) of Lemma~\ref{lemma:f_gradf_ineq}. The definition of $r$ and (A6) imply that item~(iv) is satisfied. Thus, it is sufficient to show items~(iii) and (v).
\par
Since $\{ x_{k} \} \subset {\cal S}$ holds, we have from (A3) that $\Vert x_{k} \Vert \leq R$ for $k \in \mathbb{N} \cup \{ 0 \}$. 
Let us take $k \in \mathbb{N} \cup \{ 0 \}$ arbitrarily. Then, it is clear that $\Vert x_{k} - x^{\ast} \Vert \leq D$, and hence the convexity of $f$ yields $f(x_{k}) - f^{\ast} \leq \langle \nabla f(x_{k}), x_{k} - x^{\ast} \rangle \leq D \Vert \nabla f(x_{k}) \Vert$. This fact implies that item~(iii) holds, and
\begin{align}
\left( \frac{ f(x_{k}) - f^{\ast} }{D} \right)^{\frac{3}{2}} \leq \Vert \nabla f(x_{k}) \Vert^{\frac{3}{2}}. \label{ineq:f_2/3}
\end{align}
Now, we take arbitrary $k \in {\cal I}(\theta)$. The definition of $\mu_{k}$ and~(b) of Lemma~\ref{lemma:f_gradf_ineq} derive
\begin{align*}
\theta \Vert \nabla f(x_{k}) \Vert \leq \left( 1 + (1 + q)^{\frac{3-p}{p-1}} \right) c_{1}^{\frac{p-1}{2}} \Vert \nabla f(x_{k}) \Vert^{\frac{3-p}{2}} \Vert d_{k} \Vert^{p-1} + q \Vert \nabla f(x_{k}) \Vert, 
\end{align*}
which implies
\begin{align*}
\frac{(\theta - q)^{\frac{p}{p-1}}}{c_{1}^{\frac{p}{2}} \left( 1 + (1 + q)^{\frac{3-p}{p-1}} \right)^{\frac{p}{p-1}}} \Vert \nabla f(x_{k}) \Vert^{\frac{p}{2}} \leq \Vert d_{k} \Vert^{p}.
\end{align*}
Multiplying both sides of this inequality by $\mu_{k} = c_{1}^{\frac{p-1}{2}} \Vert \nabla f(x_{k}) \Vert^{\frac{3-p}{2}}$ yields
\begin{align}
\frac{(\theta - q)^{\frac{p}{p-1}}}{\sqrt{c_{1}} \left( 1 + (1 + q)^{\frac{3-p}{p-1}} \right)^{\frac{p}{p-1}}} \Vert \nabla f(x_{k}) \Vert^{\frac{3}{2}} \leq \mu_{k} \Vert d_{k} \Vert^{p}. \label{ineq:nabf_mu_d}
\end{align}
Using item~(a) of Lemma~\ref{lemma:f_gradf_ineq} and \eqref{ineq:nabf_mu_d}, we obtain
\begin{align}
f(x_{k+1}) - f(x_{k}) \leq - \frac{ (\theta - q)^{\frac{p}{p-1}} \left( 3 - (1 + q)^{\frac{3-p}{p-1}} \right) }{3\sqrt{c_{1}} \left( 1 + (1 + q)^{\frac{3-p}{p-1}} \right)^{\frac{p}{p-1}}} \Vert \nabla f(x_{k}) \Vert^{\frac{3}{2}}. \label{ineq:fnew_f_nabf}
\end{align}
Moreover, from~\eqref{ineq:f_2/3} and \eqref{ineq:fnew_f_nabf},
\begin{align*}
f(x_{k+1}) - f(x_{k}) \leq - \tau (f(x_{k}) - f^{\ast})^{\frac{3}{2}} ~ \quad \forall k \in {\cal I}(\theta).
\end{align*}
Therefore, we can verify that item~(v) of Proposition~\ref{prop:global_convergence} holds.
\end{proof}

\noindent
From Theorem~\ref{th:global}, we have to indicate the existence of $q$ and $\theta$ satisfying (A4)--(A6) to show global ${\cal O}(k^{-2})$ convergence of Algorithm~\ref{algo:GRNM}. Although the existence of these parameters cannot be ensured for all $p > 1$, we can show their existence for specific $p \in (1, 3]$. Two examples of these concrete parameters are presented.

\begin{example} \label{ex:1}
\begin{align*}
q \coloneqq 0, \quad \theta \coloneqq \frac{3}{8}.
\end{align*} 
\end{example}

\begin{example} \label{ex:2}
\begin{align*}
q \coloneqq \min \left\{ \frac{1}{10} (2^{\frac{p-1}{3-p}} - 1), \frac{1}{20} 2^{\frac{3-p}{p-1}} \right\}, \quad
\theta \coloneqq \frac{1}{5}.
\end{align*} 
\end{example}

\noindent
The parameters $q$ and $\theta$ described in Examples~\ref{ex:1} and \ref{ex:2} satisfy conditions~(A4)--(A6). For details, see Appendix~\ref{app:param}.

\begin{remark}
When $(p, q, \theta) = (2, 0, 4^{-1})$, we can easily verify that $\tau = (96 D^{3/2} \sqrt{c_{1}})^{-1}$. This value coincides with that of Mishchenko~{\rm \cite{Mi23}}, implying that the proposed method is a generalization of~{\rm \cite{Mi23}}.
\end{remark}

\begin{remark}
Parameter $\theta$ described in Theorem~{\rm \ref{th:global}} is only required for the proof and is unrelated to problem~\eqref{problem:convex} and Algorithm~{\rm \ref{algo:GRNM}}. Hence, it should be selected to provide a good coefficient regarding the convergence rate. Since the coefficients are determined by
\begin{align*}
\frac{1}{\theta^{(\widehat{i} + 1)}}, \quad \left[ \frac{3 \sqrt{c_{1}} D^{\frac{3}{2}} \left( 1 + (1 + q)^{\frac{3-p}{p-1}} \right)^{\frac{p}{p-1}}}{(\theta - q)^{\frac{p}{p-1}} \left( 3 - (1 + q)^{\frac{3-p}{p-1}} \right)} \right]^2,
\end{align*}
we should take $\theta$ as large as possible.
\end{remark}

\section{Local superlinear convergence of Algorithm~\ref{algo:GRNM}} \label{sec:local_convergence}
This section aims to show local and superlinear convergence of Algorithm~\ref{algo:GRNM}. Throughout this section, the set of optimal solutions is denoted by $X^{\ast} \subset \mathbb{R}^{n}$. We first make an additional assumption.
\begin{description}
\item[(A7)] There exist $r_{1} > 0$ and $m_{1} > 0$ such that ${\rm dist}(x, X^{\ast}) \leq m_{1} \Vert \nabla f(x) \Vert$ for each $x \in B(x^{\ast}, r_{1})$.
\end{description}
\noindent
For a given point $x \in \mathbb{R}^{n}$, let $\widehat{x}$ be a point satisfying
\begin{align*}
\widehat{x} \in X^{\ast}, \quad \Vert \widehat{x} - x \Vert = {\rm dist}(x, X^{\ast}).
\end{align*}
\noindent
Some important inequalities for local convergence are as follows:

\begin{lemma} \label{lemma:d_dist}
Suppose that {\rm (A1), (A2), and (A7)} hold. Then, there exist $r_{2} > 0$, $m_{2} > 0$, $m_{3} > 0$, and $m_{4} > 0$ such that
\begin{description}
\item[{\rm (a)}] $\Vert d_{k} \Vert \leq m_{2}{\rm dist}(x_{k}, X^{\ast})$ and $\Vert d_{k} \Vert \geq m_{3} {\rm dist}(x_{k}, X^{\ast})$ for $x_{k} \in B(x^{\ast}, r_{2})$;
\item[{\rm (b)}] ${\rm dist}(x_{k+1}, X^{\ast}) \leq m_{4} {\rm dist}(x_{k}, X^{\ast})^{\frac{p+1}{2}}$ for $x_{k}, \, x_{k+1} \in B(x^{\ast}, r_{2})$.
\end{description}
\end{lemma}

\begin{proof}
To begin with, we define $u_{1}$ and $u_{2}$ as
\begin{align*}
u_{1} \coloneqq \sup \left\{ \Vert \nabla^{2} f(z) \Vert; z \in B(x^{\ast}, r_{1}) \right\}, \quad u_{2} \coloneqq u_{1} + \frac{c_{1} r_{1}}{2},
\end{align*}
respectively, and will show the following inequality:
\begin{align}
\Vert \nabla f(x) \Vert \leq u_{2} {\rm dist}(x, X^{\ast}) \quad \forall x \in B(x^{\ast}, r_{1}). \label{ineq:nabla_f_dist}
\end{align}
We have $\Vert \widehat{x} - x \Vert \leq \Vert x - x^{\ast} \Vert \leq r_{1}$. Hence, the first inequality of (A1) guarantees $\Vert \nabla f(\widehat{x}) - \nabla f(x) - \nabla^{2} f(x) (\widehat{x} - x) \Vert \leq \frac{c_{1}}{2} {\rm dist}(x, X^{\ast})^{2}$. It then follows from $\nabla f(\widehat{x}) = 0$ and ${\rm dist}(x, X^{\ast}) \leq \Vert x - x^{\ast}  \Vert \leq r_{1}$ that $\Vert \nabla f(x) \Vert - u_{1} {\rm dist}(x, X^{\ast}) \leq \frac{c_{1} r_{1}}{2} {\rm dist}(x, X^{\ast})$. Thus, inequality~\eqref{ineq:nabla_f_dist} holds.
\par
We show item~(a). Define $r_{2}$ as follows: 
\begin{align*}
r_{2} \coloneqq \min \left\{ r_{1}, \left( 2 \sqrt{n} c_{2} u_{2}^{\frac{p-1}{2}} \right)^{-\frac{2}{p-1}} \right\}. 
\end{align*}
Let $x_{k} \in B(x^{\ast}, r_{2})$. The definition of ${\rm dist}(x_{k}, X^{\ast})$ and inequality~\eqref{ineq:nabla_f_dist} ensure
\begin{align}
& {\rm dist}(x_{k}, X^{\ast}) \leq \Vert x_{k} - x^{\ast} \Vert \leq r_{2} \label{ineq:dist_r}
\\
& \Vert \nabla f(x_{k}) \Vert \leq u_{2} {\rm dist}(x_{k}, X^{\ast}) \leq u_{2} \Vert x_{k} - x^{\ast} \Vert \leq u_{2} r_{2}. \label{ineq:dist_mr}
\end{align}
Lemma~\ref{lemma:solvable} and~\eqref{two_ineq:mat_H} lead to $\Vert d_{k} \Vert \leq \Vert ( \nabla^{2} f(x_{k}) + \mu_{k} \Vert d_{k} \Vert^{p-2} I )^{-1} \nabla f(x_{k}) \Vert + \frac{\sqrt{n} \rho_{k}}{\mu_{k} \Vert d_{k} \Vert^{p-2}}$.
It then follows from (A2) and~\eqref{ineq:nabla_f_dist} that
\begin{align}
\Vert d_{k} \Vert 
&\leq \left\Vert ( \nabla^{2} f(x_{k}) + \mu_{k} \Vert d_{k} \Vert^{p-2} I )^{-1} \nabla f(x_{k}) \right\Vert \nonumber
\\
&\hspace{40mm} 
+ \frac{\sqrt{n} c_{2} u_{2}^{\frac{p+1}{2}} }{\mu_{k} \Vert d_{k} \Vert^{p-2}} {\rm dist}(x_{k}, X^{\ast})^{\frac{p+1}{2}}. \label{ineq:dk_dist1}
\end{align}
Now, we notice that $\nabla f(x_{k}) = - (\nabla f(\widehat{x}_{k}) - \nabla f(x_{k}) - \nabla^{2} f(x_{k}) (\widehat{x}_{k} - x_{k}) ) - \nabla^{2} f(x_{k}) (\widehat{x}_{k} - x_{k})$ holds from $\nabla f(\widehat{x}_{k}) = 0$. Then, combining~(A1), \eqref{two_ineq:mat_H}, and~\eqref{two_ineq:mat_H2} implies
\begin{align}
& \left\Vert ( \nabla^{2} f(x_{k}) + \mu_{k} \Vert d_{k} \Vert^{p-2} I )^{-1} \nabla f(x_{k}) \right\Vert \nonumber
\\
&\leq \left\Vert (\nabla^{2} f(x_{k}) + \mu_{k} \Vert d_{k} \Vert^{p-2} I)^{-1} \right\Vert \Vert \nabla f(\widehat{x}_{k}) - \nabla f(x_{k}) - \nabla^{2} f(x_{k}) (\widehat{x}_{k} - x_{k}) \Vert \nonumber
\\
& \hspace{35mm} + \left\Vert (\nabla^{2} f(x_{k}) + \mu_{k} \Vert d_{k} \Vert^{p-2} I)^{-1} \nabla^{2} f(x_{k}) \right\Vert \Vert \widehat{x}_{k} - x_{k} \Vert  \nonumber
\\
&\leq \frac{c_{1}}{2\mu_{k}\Vert d_{k} \Vert^{p-2}} {\rm dist}(x_{k}, X^{\ast})^{2} + {\rm dist}(x_{k}, X^{\ast}). \label{ineq:nablaf_dist2}
\end{align}
By \eqref{ineq:dist_r}, \eqref{ineq:dk_dist1} and \eqref{ineq:nablaf_dist2}, we get $\Vert d_{k} \Vert \leq \frac{u_{3}}{\mu_{k} \Vert d_{k} \Vert^{p-2}} {\rm dist}(x_{k}, X^{\ast})^{\frac{p+1}{2}} + {\rm dist}(x_{k}, X^{\ast})$, where $u_{3} \coloneqq \frac{1}{2}c_{1}r_{2}^{ \frac{3-p}{2} } + \sqrt{n} c_{2} u_{2}^{ \frac{p+1}{2} }$.
This inequality can be reformulated as $\mu_{k} \Vert d_{k} \Vert^{p-1} \leq 2 \max \{ u_{3} {\rm dist}(x_{k}, X^{\ast})^{\frac{p+1}{2}}, \mu_{k} \Vert d_{k} \Vert^{p-2} {\rm dist}(x_{k}, X^{\ast}) \}$. There are two possible cases: (i) $\mu_{k} \Vert d_{k} \Vert^{p-1} \leq 2u_{3} {\rm dist}(x_{k}, X^{\ast})^{\frac{p+1}{2}}$; (ii) $\mu_{k} \Vert d_{k} \Vert^{p-1} \leq 2\mu_{k} \Vert d_{k} \Vert^{p-2} {\rm dist}(x_{k}, X^{\ast})$. In case~(i), utilizing (A2) and (A7) derives
\begin{align*}
\frac{c_{1}^{\frac{p-1}{2}}{\rm dist}(x_{k}, X^{\ast})^{\frac{3-p}{2}}}{m_{1}^{\frac{3-p}{2}}} \Vert d_{k} \Vert^{p-1} \leq \mu_{k} \Vert d_{k} \Vert^{p-1} \leq 2u_{3} {\rm dist}(x_{k}, X^{\ast})^{\frac{p+1}{2}}.
\end{align*}
Hence, we have $\Vert d_{k} \Vert \leq c_{1}^{-\frac{1}{2}}(2u_{3}m_{1}^{\frac{3-p}{2}})^{\frac{1}{p-1}} {\rm dist}(x_{k}, X^{\ast}).$
Meanwhile, case~(ii) leads to $\Vert d_{k} \Vert \leq 2 {\rm dist}(x_{k}, X^{\ast})$. Thus, there exists $m_{2} > 0$ such that $\Vert d_{k} \Vert \leq m_{2} {\rm dist}(x_{k}, X^{\ast})$.
\par
Now, Lemma~\ref{lemma:solvable}, (A2), and \eqref{ineq:dist_mr} yield
\begin{align}
\Vert d_{k} \Vert
&\geq \frac{\Vert \nabla f(x_{k}) + \rho_{k} \eta_{k} \Vert}{\Vert \nabla^{2} f(x_{k}) + \mu_{k} \Vert d_{k} \Vert^{p-2} I \Vert} \nonumber
\\
&\geq \frac{1 - \sqrt{n}c_{2} (u_{2} r_{2})^{\frac{p-1}{2}} }{u_{1} + \mu_{k} \Vert d_{k} \Vert^{p-2}} \Vert \nabla f(x_{k}) \Vert \nonumber
\\
&\geq \frac{1}{2 (u_{1} + \mu_{k} \Vert d_{k} \Vert^{p-2})} \Vert \nabla f(x_{k}) \Vert, \label{ineq:d_nablaf}
\end{align}
where the last inequality follows from the definition of $r_{2}$. Exploiting~(A2) and $\Vert d_{k} \Vert \leq m_{2} {\rm dist}(x_{k}, X^{\ast})$ derives
\begin{align*}
\mu_{k} \Vert d_{k} \Vert^{p-2} \leq c_{1}^{\frac{p-1}{2}} \Vert \nabla f(x_{k}) \Vert^{\frac{3-p}{2}} m_{2}^{p-2} {\rm dist}(x_{k}, X^{\ast})^{p-2}.
\end{align*}
It then follows from \eqref{ineq:dist_r} and \eqref{ineq:dist_mr} that
\begin{align}
\mu_{k} \Vert d_{k} \Vert^{p-2} \leq c_{1}^{\frac{p-1}{2}} u_{2}^{\frac{3-p}{2}} m_{2}^{p-2} r_{2}^{\frac{p-1}{2}}. \label{ineq:mud_m}
\end{align}
By~\eqref{ineq:d_nablaf}, \eqref{ineq:mud_m}, and (A7), there exists $m_{3} > 0$ satisfying $\Vert d_{k} \Vert \geq m_{3} {\rm dist}(x_{k}, X^{\ast})$.
\par
Next, we prove item~(b). Let $x_{k+1} = x_{k} + d_{k} \in B(x^{\ast}, r_{2})$. Lemma~\ref{lemma:solvable} implies
\begin{align*}
\nabla f(x_{k+1}) &= (\nabla f(x_{k} + d_{k}) - \nabla f(x_{k}) - \nabla^{2} f(x_{k}) d_{k}) - (\mu_{k} \Vert d_{k} \Vert^{p-2} d_{k} + \rho_{k} \eta_{k}) \nonumber
\\
& \hspace{25.75mm} + ( \nabla f(x_{k}) + (\nabla^{2} f(x_{k}) + \mu_{k} \Vert d_{k} \Vert^{p-2} I)d_{k} + \rho_{k} \eta_{k} )  \nonumber
\\
&= (\nabla f(x_{k} + d_{k}) - \nabla f(x_{k}) - \nabla^{2} f(x_{k}) d_{k}) - (\mu_{k} \Vert d_{k} \Vert^{p-2} d_{k} + \rho_{k} \eta_{k}). 
\end{align*}
Then, we have from (A1), (A2), and (A7) that 
\begin{align} 
&\hspace{-1.3mm} {\rm dist}(x_{k+1}, X^{\ast}) \nonumber
\\
&\hspace{-1.3mm} \leq m_{1} \Vert \nabla f(x_{k} + d_{k}) \Vert \nonumber
\\
&\hspace{-1.3mm} \leq m_{1} \Vert \nabla f(x_{k} + d_{k}) - \nabla f(x_{k}) - \nabla^{2} f(x_{k}) d_{k} \Vert + m_{1} \mu_{k} \Vert d_{k} \Vert^{p-1} + \sqrt{n} m_{1}\rho_{k} \nonumber
\\
&\hspace{-1.3mm} \leq \frac{c_{1}m_{1}}{2} \Vert d_{k} \Vert^{2} + c_{1}^{\frac{p+1}{2}} m_{1} \Vert \nabla f(x_{k}) \Vert^{\frac{3-p}{2}} \Vert d_{k} \Vert^{p-1} + \sqrt{n} c_{2} m_{1} \Vert \nabla f(x_{k}) \Vert^{\frac{p+1}{2}}.  \label{ineq:distk_distk}
\end{align}
Now, recall that $\Vert d_{k} \Vert \leq m_{2} {\rm dist}(x_{k}, X^{\ast})$, $\Vert \nabla f(x_{k}) \Vert \leq u_{2} {\rm dist}(x_{k}, X^{\ast})$, and ${\rm dist}(x_{k}, X^{\ast}) \leq r_{2}$, where the second and third inequalities follow from~\eqref{ineq:nabla_f_dist} and~\eqref{ineq:dist_r}, respectively. Thus, we can easily verify that
\begin{align}
&\hspace{-1.3mm} \frac{c_{1}m_{1}}{2} \Vert d_{k} \Vert^{2} + c_{1}^{\frac{p+1}{2}} m_{1} \Vert \nabla f(x_{k}) \Vert^{\frac{3-p}{2}} \Vert d_{k} \Vert^{p-1} + \sqrt{n} c_{2} m_{1} \Vert \nabla f(x_{k}) \Vert^{\frac{p+1}{2}} \nonumber
\\
&\hspace{-1.3mm} \leq \left( \frac{c_{1}m_{1}m_{2}^{2}r_{2}^{\frac{3-p}{2}}}{2} + c_{1}^{\frac{p+1}{2}} u_{2}^{\frac{3-p}{2}} m_{2}^{p-1} + \sqrt{n} c_{2} m_{1} u_{2}^{\frac{p+1}{2}} \right) {\rm dist}(x_{k}, X^{\ast})^{\frac{p+1}{2}}. \label{ineq:distk_distk2}
\end{align}
Therefore, combining~\eqref{ineq:distk_distk} and \eqref{ineq:distk_distk2} guarantees the existence of $m_{4} > 0$ satisfying ${\rm dist}(x_{k+1}, X^{\ast}) \leq m_{4} {\rm dist}(x_{k}, X^{\ast})^{\frac{p+1}{2}}$.
\end{proof}

\noindent
Finally, we establish local and superlinear convergence of Algorithm~\ref{algo:GRNM}. Although we can prove the theorem using the above lemmas in a manner similar to~\cite[Theorem~3.2]{LiFuQiYa04}, the proof is given in Appendix~\ref{app:proof} for completeness of the paper.

\begin{theorem} \label{th:local}
Suppose that {\rm (A1), (A2), and (A7)} hold. If an initial point $x_{0}$ is chosen sufficiently close to $x^{\ast}$, then any sequence $\{ x_{k} \}$ generated by Algorithm~{\rm \ref{algo:GRNM}} converges to some global optimum $\bar{x} \in X^{\ast}$ superlinearly. Moreover, if $p=3$ is satisfied, then $\{ x_{k} \}$ converges to $\bar{x} \in X^{\ast}$ quadratically.
\end{theorem}

\section{Concluding remarks} \label{sec:conclusion}
In this paper, we have proposed Algorithm~\ref{algo:GRNM}, which is an RNM with generalized regularization terms. The proposed method is based on the RNM proposed by Mishchenko~\cite{Mi23}, but it is a generalization of the existing one regarding regularization. Therefore, not only the quadratic and cubic RNMs but also novel RNMs with other regularization, such as the elastic net, are included in Algorithm~\ref{algo:GRNM}. We have proven global $\mathcal{O}(k^{-2})$ and local superlinear convergence of~Algorithm~\ref{algo:GRNM}.
\par
One of future research is to propose an accelerated GRNM that globally converges in the order of ${\cal O}(k^{-3})$.


\section*{Declarations}
\noindent
{\bf Conflict of interest}~~The authors declare no conflicts of interest.

\bibliographystyle{elsarticle-num}
\bibliography{reference}

\begin{thebibliography}{10}
\expandafter\ifx\csname url\endcsname\relax
  \def\url#1{\texttt{#1}}\fi
\expandafter\ifx\csname urlprefix\endcsname\relax\def\urlprefix{URL }\fi
\expandafter\ifx\csname href\endcsname\relax
  \def\href#1#2{#2} \def\path#1{#1}\fi

\bibitem{DoMiNe24}
N.~Doikov, K.~Mishchenko, Y.~Nesterov, Super-universal regularized {N}ewton
  method, SIAM Journal on optimization 34 (2024) 27--56.

\bibitem{LiFuQiYa04}
D.~Li, M.~Fukushima, L.~Qi, N.~Yamashita, Regularized {N}ewton methods for
  convex minimization problems with singular solutions, Computational
  Optimization and Applications 28 (2004) 131--147.

\bibitem{LiLi09}
Y.-J. Li, D.-H. Li, Truncated regularized {N}ewton method for convex
  minimizations, Computational Optimization and Applications 43 (2009)
  119--131.

\bibitem{Po09}
R.~A. Polyak, Regularized {N}ewton method for unconstrained convex
  optimization, Mathematical Programming 120 (2009) 125--145.

\bibitem{Mi23}
K.~Mishchenko, Regularized {N}ewton method with global $\mathcal{O}(1/k^{2})$
  convergence, SIAM Journal on optimization 33 (2023) 1440--1462.

\bibitem{NePo06}
Y.~Nesterov, B.~T. Polyak, Cubic regularization of {N}ewton method and its
  global performance, Mathematical Programming 108 (2006) 177--205.

\bibitem{Ne08}
Y.~Nesterov, Accelerating the cubic regularization of newton's method on convex
  problems, Mathematical Programming 112 (2008) 159--181.

\bibitem{YuZhSo19}
M.-C. Yue, Z.~Zhou, A.~M.-C. So, On the quadratic convergence of the cubic
  regularization method under a local error bound condition, SIAM Journal on
  Optimization 29 (2019) 90--932.

\bibitem{GoReBa20}
D.~Goldfarb, Y.~Ren, A.~Bahamou, Practical quasi-{N}ewton methods for training
  deep neural networks, in: Advances in Neural Information Processing Systems,
  2020, pp. 2386--2396.

\bibitem{GrLaLu86}
L.~Grippo, F.~Lampariello, S.~Luclidi, A nonmonotone line search technique for
  {N}ewton's method, SIAM Journal on Numerical Analysis 23 (1986) 707--716.

\bibitem{RoNe21}
A.~Rodomanov, Y.~Nesterov, Greedy quasi-{N}ewton methods with explicit
  superlinear convergence, SIAM Journal on Optimization 31 (2021) 785--811.

\bibitem{CrRo20}
R.~Crane, F.~Roosta, {DINO}: {D}istributed {N}ewton-type optimization method,
  in: Proceedings of the 37th International Conference on Machine Learning,
  2020, pp. 2174--2184.

\bibitem{DoRi18}
N.~Doikov, P.~Richt\'{a}rik, Randomized block cubic {N}ewton method, in:
  Proceedings of the 35th International Conference on Machine Learning, PMLR,
  2018, pp. 1290--1298.

\bibitem{GoKoLiRi19}
R.~M. Gower, D.~Kovalev, F.~Lieder, P.~Richt\'{a}rik, {RSN}: {R}andomized
  subspace {N}ewton, in: Advances in Neural Information Processing Systems,
  2019, pp. 616--625.

\bibitem{HaDoRiNe20}
F.~Hanzely, N.~Doikov, P.~Richt\'{a}rik, Y.~Nesterov, {S}tochastic subspace
  cubic {N}ewton method, in: Proceedings of the 37th International Conference
  on Machine Learning, 2020, pp. 4027--4038.

\bibitem{MaBaRu19}
U.~Marteau-Ferey, F.~Bach, A.~Rudi, Globally convergent {N}ewton methods for
  ill-conditioned generalized self-concordant losses, in: Advances in Neural
  Information Processing Systems, 2019, pp. 7636--7646.

\bibitem{RoKr16}
A.~Rodomanov, D.~Kropotov, A superlinearly-convergent proximal {N}ewton-type
  method for the optimization of finite sums, in: Proceedings of the
  International Conference on Machine Learning, PRML, 2016, pp. 2597--2605.

\bibitem{SoMiMoDeGu20}
S.~Soori, K.~Mishchenko, A.~Mokhtari, M.~M. Dehnavi, M.~G\"{u}rb\"{u}zbalaban,
  {DA}ve-{QN}: {A} distributed averaged quasi-{N}ewton method with local
  superlinear convergence rate, in: Proceedings of the International Conference
  on Artificial Intelligence and Statistics, 2020, pp. 1965--1976.

\bibitem{ArYaYa24}
S.~Ariizumi, Y.~Yamakawa, N.~Yamashita, Convergence properties of
  {L}evenberg-{M}arquardt methods with generalized regularization terms,
  Applied Mathematics and Computation 463 (2024) 128365.

\bibitem{Be:ConOpt}
D.~P. Bertsekas, {C}onvex {O}ptimization {A}lgorithms, Athena Scientific,
  Massachusetts, 2015.

\end{thebibliography}

\appendix
\makeatletter
\setcounter{equation}{0}
\renewcommand{\theequation}{A.\arabic{equation}}

\section{} \label{app:param}
In this appendix, we show that the parameters described in Examples~\ref{ex:1} and~\ref{ex:2} satisfy assumptions~(A4)--(A6) stated in Theorem~\ref{th:global}.
\par
We now discuss Example~\ref{ex:1}. It can be verified that 
\begin{gather*}
q = 0, ~~ 3 - (1 + q)^{\frac{3-p}{p-1}} = 3, ~~ 1 + 2 q + (1 + q)^{\frac{2}{p-1}} = 2, 
\\
\left( 1 + 2 q + (1 + q)^{\frac{2}{p-1}} \right) \theta = \frac{3}{4}.
\end{gather*}
Thus, we can easily verify that assumptions~(A4) and (A5) hold. Moreover, assumption (A6) is obtained from
\begin{align*}
&\theta^{k} = \left( \frac{3}{8} \right)^{k} = {\cal O}(k^{-2}) ~~ (k \to \infty),
\\
&\left( \left( 1 + 2 q + (1 + q)^{\frac{2}{p-1}} \right) \theta \right)^{\frac{k}{2}} = \left( \frac{3}{4} \right)^{\frac{k}{2}} = {\cal O}(k^{-2}) ~~ (k \to \infty).
\end{align*}
\par
Next, we consider Example~\ref{ex:2}. In this case, the parameter $q$ depends on $p \in (1,3]$, and hence we denote $q = q(p)$. Moreover, we use the following notation:
\begin{align*}
s(p) \coloneqq 3 - (1 + q(p))^{\frac{3-p}{p-1}}, \quad t(p) \coloneqq 1 + 2 q(p) + (1 + q(p))^{\frac{2}{p-1}}.
\end{align*}
For any $x \in (1,3)$, we define
\begin{align*}
& q_{1}(x) \coloneqq \frac{1}{10} ( 2^{\frac{x-1}{3-x}} - 1 ) > 0, && t_{1}(x) \coloneqq 1 + 2q_{1}(x) + ( 1 + q_{1}(x) )^{\frac{2}{x-1}},
\\
& q_{2}(x) \coloneqq \frac{1}{20} 2^{\frac{3-x}{x-1}} > 0, && t_{2}(x) \coloneqq 1 + 2q_{2}(x) + ( 1 + q_{2}(x) )^{\frac{2}{x-1}}.
\end{align*}
Recall that $q(p)$ and $t(p)$ can be represented as follows:
\begin{align} \label{def:alpha_gamma}
q(p) = \left\{
\begin{aligned}
& q_{1}(p) && {\rm if~} p \in (1,2],
\\
& q_{2}(p) && {\rm if~} p \in (2,3],
\end{aligned}
\right. \qquad t(p) = \left\{
\begin{aligned}
& t_{1}(p) && {\rm if~} p \in (1,2],
\\
& t_{2}(p) && {\rm if~} p \in (2,3].
\end{aligned}
\right.
\end{align}
By the definitions of $q_{1}$ and $q_{2}$, we obtain
\begin{align} \label{def:derivative_alpha}
\begin{aligned}
& \frac{d}{d p}q_{1}(p) = \frac{1}{10 (3-p)^2} 2^{\frac{2}{3-p}} \log 2 > 0 \quad \forall p \in (1,2],
\\
& \frac{d}{d p}q_{2}(p) = -\frac{1}{20 (p-1)^2} 2^{\frac{2}{p-1}} \log 2 < 0 \quad \forall p \in (2,3].
\end{aligned}
\end{align}
Thus, the first equality of~\eqref{def:alpha_gamma} implies that $q$ is monotonically increasing for $p \in (1,2]$ and is monotonically decreasing for $p \in (2,3]$. Thus, using $\theta = \frac{1}{5}$ yields 
\begin{align}
0 < q(p) \leq q(2) = \frac{1}{10} < \theta < 1 \quad \forall p \in (1,3]. \label{ineq:alpha_p}
\end{align}
Since $q(p) = \min \{ q_{1}(p), q_{2}(p) \} \leq q_{1}(p) = \frac{1}{10} (2^{\frac{p-1}{3-p}} - 1) < 2^{\frac{p-1}{3-p}} - 1$ for $p \in (1,3]$, we have 
\begin{align}
s(p) = 3 - (1 + q(p))^{\frac{3-p}{p-1}} > 3 - 2 = 1 > 0 \quad \forall p \in (1,3]. \label{ineq:beta_p}
\end{align}
Noting $\frac{2}{p-1} \geq 1$ and $q(p) > 0$ derives
\begin{align}
t(p) = 1 + 2 q(p) + (1 + q(p))^{\frac{2}{p-1}} \geq 2 + 3q(p) \geq 1 \quad \forall p \in (1,3]. \label{ineq:gamma_p}
\end{align}
Utilizing~\eqref{def:derivative_alpha} implies
\begin{align*}
\begin{aligned}
&\frac{d}{d p} t_{1}(p) = \left( 2 + \frac{2}{p-1} (1 + q_{1}(p))^{\frac{3-p}{p-1}} \right) \frac{d}{dp}q_{1}(p) > 0 \quad \forall p \in (1,2],
\\
&\frac{d}{d p} t_{2}(p) = \left( 2 + \frac{2}{p-1} (1 + q_{2}(p))^{\frac{3-p}{p-1}} \right) \frac{d}{dp}q_{2}(p) < 0 \quad \forall p \in (2,3].
\end{aligned}
\end{align*}
Hence, the second equality of~\eqref{def:alpha_gamma} derives $t(p) \leq t(2) = \frac{241}{100}$ for $p \in (1,3]$. It then follows from~\eqref{ineq:gamma_p} and $\theta = \frac{1}{5}$ that
\begin{align}
0 < \frac{1}{5} \leq \theta t(p) \leq \frac{241}{500} < 1. \label{ineq:delta_gamma_p}
\end{align}
Therefore, the assumptions are ensured by~\eqref{ineq:alpha_p}, \eqref{ineq:beta_p}, \eqref{ineq:gamma_p}, and \eqref{ineq:delta_gamma_p}.

\makeatletter
\setcounter{equation}{0}
\renewcommand{\theequation}{B.\arabic{equation}}
\section{} \label{app:proof}
This appendix provides the proof of Theorem~\ref{th:local}.
\begin{proof}
We define $r_{3}$ and $r_{4}$ as follows:
\begin{align*}
r_{3} \coloneqq \min \left\{ r_{2}, \left( \frac{m_{3}}{3m_{2} m_{4}} \right)^{\frac{2}{p-1}} \right\}, ~~ r_{4} \coloneqq \frac{1}{2+m_{2}} \min \left\{ r_{3}, \left(\frac{m_{3}}{3m_{2}m_{4}}\right)^{\frac{2}{p-1}} \right\},
\end{align*}
where $r_{2}$, $m_{2}$, $m_{3}$, and $m_{4}$ are positive constants described in Lemma~\ref{lemma:d_dist}. Assume that the initial point $x_{0}$ is selected from $B(x^{\ast}, r_{4})$, thus it satisfies $\Vert x_{0} - x^{\ast} \Vert \leq r_{4}$.
\par
The proof is divided into two parts: The former part will prepare two inequalities regarding $\Vert d_{k} \Vert$, and the latter part will prove fast convergence of $\{ x_{k} \}$ using those inequalities. We first show
\begin{align}
\Vert d_{k} \Vert \leq \frac{1}{3} \Vert d_{k-1} \Vert, \quad \Vert d_{k} \Vert \leq \frac{r_{4} m_{2}}{3^{k}} \quad \forall k \in \mathbb{N}. \label{ineq:two_dk}
\end{align}
Let $k \in \mathbb{N}$ be arbitrary. Using (a) and (b) of Lemma~\ref{lemma:d_dist} yields
\begin{align}
\Vert d_{k} \Vert \leq m_{2} m_{4} {\rm dist}(x_{k-1}, X^{\ast})^{\frac{p+1}{2}} \leq \frac{m_{2} m_{4}}{m_{3}} r_{3}^{\frac{p-1}{2}} \Vert d_{k-1} \Vert \leq \frac{1}{3} \Vert d_{k-1} \Vert, \label{ineq:dk_frac12}
\end{align}
where note that the last inequality follows from the definition of $r_{3}$. Now, let us show that
\begin{align} \label{d_converges_0}
x_{\ell} \in B(x^{\ast}, r_{3}) \,~ \forall \ell \in \{ 0,1, \ldots, k \} \,~ \Longrightarrow \,~ \Vert d_{\ell} \Vert \leq \frac{r_{4} m_{2}}{3^{\ell}} \,~ \forall \ell \in \{ 0,1, \ldots, k \}.
\end{align}
From \eqref{ineq:dk_frac12} and (a) of Lemma~\ref{lemma:d_dist}, we have
\begin{align*}
& \Vert d_{0} \Vert \leq m_{2} {\rm dist}(x_{0}, X^{\ast}) \leq m_{2} \Vert x_{0} - x^{\ast} \Vert \leq r_{4} m_{2},
\\
& \Vert d_{\ell} \Vert \leq \frac{1}{3} \Vert d_{\ell-1} \Vert \leq \cdots \leq \frac{1}{3^{\ell}} \Vert d_{0} \Vert \leq \frac{r_{4} m_{2}}{3^{\ell}} \quad \forall \ell \in \{ 1, 2, \ldots, k \},
\end{align*}
namely, \eqref{d_converges_0} can be verified. 
\par
From now on, we prove by mathematical induction that $x_{k} \in B(x^{\ast}, r_{3})$ for all $k \in \mathbb{N}$. Let us consider the case where $k=1$. Item~(a) of Lemma~\ref{lemma:d_dist} implies $\Vert x_{1} - x^{\ast} \Vert \leq \Vert x_{0} - x^{\ast} \Vert + \Vert d_{0} \Vert \leq r_{4} + m_{2} {\rm dist}(x_{0}, X^{\ast}) \leq r_{4} (1+m_{2}) \leq r_{3}$. Next, let $k \in \mathbb{N}$ be arbitrary, and we assume that $x_{j} \in B(x^{\ast}, r_{3})$ for $j \in \{ 0,1, \ldots, k \}$. By \eqref{d_converges_0} and item~(a) of Lemma~\ref{lemma:d_dist}, we obtain
\begin{align*}
\Vert x_{k+1} - x^{\ast} \Vert 
&\leq \Vert x_{k} - x^{\ast} \Vert + \Vert d_{k} \Vert \leq \cdots \leq \Vert x_{0} - x^{\ast} \Vert + \sum_{\ell=0}^{k} \Vert d_{\ell} \Vert 
\\
&\leq r_{4} + \frac{r_{4} m_{2}}{2} \left( 1 - \frac{1}{3^{k+1}}  \right) \leq \frac{2 + m_{2}}{2} r_{4} \leq r_{3}.
\end{align*}
Thus, we verify that $x_{k} \in B(x^{\ast}, r_{3})$ for $k \in \mathbb{N}$. It then follows from~\eqref{d_converges_0} that $\Vert d_{k} \Vert \leq \frac{r_{4} m_{2}}{3^{k}}$ for $k \in \mathbb{N}$, namely, the disired inequalities of~\eqref{ineq:two_dk} are proven.
\par
The second part shows the local fast convergence of $\{ x_{k} \}$. We arbitrarily take $i \in \mathbb{N}$ and $j \in \mathbb{N}$ with $i \gg j$. Using \eqref{ineq:two_dk} yields
\begin{align*}
\Vert x_{i} - x_{j} \Vert
&\leq \Vert x_{i-1} - x_{j} \Vert + \Vert d_{i-1} \Vert \leq \Vert x_{i-2} - x_{j} \Vert + \sum_{\ell=i-2}^{i-1} \Vert d_{\ell} \Vert
\\
&\leq \cdots \leq \sum_{\ell=j}^{i-1} \Vert d_{\ell} \Vert \leq r_{4} m_{2} \sum_{\ell=j}^{i-1} \frac{1}{3^{\ell}} = r_{4} m_{2} \left( \frac{1}{3^{j-1}} - \frac{1}{3^{i-1}}  \right) \leq \frac{r_{4} m_{2}}{3^{j-1}}.
\end{align*}
This fact implies that $\{ x_{k} \}$ is a Cuachy sequence, that is, there exists $\bar{x} \in \mathbb{R}^{n}$ such that $x_{k} \to \bar{x}$ as $k \to \infty$. Meanwhile, it follows from (a) of Lemma~\ref{lemma:d_dist} and \eqref{ineq:two_dk} that $\{ {\rm dist}(x_{k}, X^{\ast}) \}$ converges to zero. We note that $\Vert \widehat{x}_{k} \Vert \leq \Vert \widehat{x}_{k} - x_{k} \Vert + \Vert x_{k} \Vert = {\rm dist}(x_{k}, X^{\ast}) + \Vert x_{k} \Vert$, namely, $\{ \widehat{x}_{k} \}$ is bounded. Hence, there exist $\widetilde{x} \in X^{\ast}$ and ${\cal K} \subset \mathbb{N}$ such that $\widehat{x}_{k} \to \widetilde{x}$ as ${\cal K} \ni k \to \infty$. These facts imply that $\Vert \bar{x} - \widetilde{x} \Vert \leq \Vert x_{k} - \bar{x} \Vert + {\rm dist}(x_{k}, X^{\ast}) + \Vert \widehat{x}_{k} - \widetilde{x} \Vert \to 0$ as ${\cal K} \ni k \to \infty$, that is, $\{ x_{k} \}$ converges to some global optimum $\bar{x} = \widetilde{x} \in X^{\ast}$.
\par
Hereinafter, we show that $\{ x_{k} \}$ converges to $\bar{x}$ superlinearly. Combining~items (a) and (b) of Lemma~\ref{lemma:d_dist} derives 
\begin{align}
\Vert d_{j+1} \Vert \leq m_{2} {\rm dist}(x_{j+1}, X^{\ast}) \leq m_{2} m_{4} {\rm dist}(x_{j}, X^{\ast})^{\frac{p+1}{2}} \leq \frac{ m_{2} m_{4} }{m_{3}^{\frac{p+1}{2}}} \Vert d_{j} \Vert^{\frac{p+1}{2}}. \label{ineq:quad_dk}
\end{align}
Since the first inequality of~\eqref{ineq:two_dk} holds, it can be verfied that
\begin{align*}
& \Vert d_{\ell} \Vert \leq \frac{1}{3^{\ell - j - 1}} \Vert d_{j+1} \Vert \quad \forall \ell \in \{ j+1, j+2, \ldots, i-1 \},
\\
& \Vert d_{\ell} \Vert \leq \frac{1}{3^{\ell - j}} \Vert d_{j} \Vert \quad \forall \ell \in \{ j+1, j+2, \ldots, i-1 \}.
\end{align*}
By these inequalities, we obtain
\begin{align*}
\Vert x_{j+1} - x_{i} \Vert
&= \left\Vert \sum_{\ell=j+1}^{i-1} d_{\ell} \right\Vert \leq \sum_{\ell=j+1}^{i-1} \Vert d_{\ell} \Vert
\\
& \sum_{\ell=j+1}^{i-1} \frac{1}{3^{\ell-j-1}} \Vert d_{j+1} \Vert = \frac{3}{2} \left( 1 - \frac{1}{3^{i-j-1}} \right) \Vert d_{j+1} \Vert \leq 2 \Vert d_{j+1} \Vert,
\end{align*}
and
\begin{align*}
\Vert x_{j} - x_{i} \Vert 
&= \left\Vert \sum_{\ell=j}^{i-1} d_{\ell} \right\Vert \geq \Vert d_{j} \Vert - \sum_{\ell=j+1}^{i-1} \Vert d_{\ell} \Vert 
\\
&\geq \left( 1 - \sum_{\ell=j+1}^{i-1} \frac{1}{3^{\ell-j}}  \right) \Vert d_{j} \Vert = \frac{1}{2} \left( 1 + \frac{1}{3^{i-j-1}} \right) \Vert d_{j} \Vert \geq \frac{1}{2} \Vert d_{j} \Vert.
\end{align*}
Hence taking the limit $i \to \infty$ implies
\begin{align}
\Vert x_{j+1} - \bar{x} \Vert \leq 2 \Vert d_{j+1} \Vert, \quad \Vert d_{j} \Vert \leq 2 \Vert x_{j} - \bar{x} \Vert. \label{ineq:quad_xk}
\end{align}
Exploiting \eqref{ineq:quad_dk} and \eqref{ineq:quad_xk} derives
\begin{align*}
\Vert x_{j+1} - \bar{x} \Vert \leq \frac{ 2^{\frac{p+3}{2}} m_{2} m_{4} }{m_{3}^{\frac{p+1}{2}}} \Vert x_{j} - \bar{x} \Vert^{\frac{p+1}{2}}.
\end{align*}
Therefore, from $\frac{p+1}{2} \in (1,2]$, the sequence $\{ x_{k} \}$ converges to $\bar{x}$ superlinearly. Moreover, if $p=3$ holds, then $\frac{p+1}{2} = 2$, that is, the rate of convergence is quadratic.
\end{proof}

\end{document}